\documentclass[leqno,12pt]{amsart} 
\usepackage{amssymb,enumerate}
\usepackage{pgf,tikz}
\usepackage{mathrsfs}
\usetikzlibrary{arrows}

\theoremstyle{plain} 
\newtheorem{theorem}{\indent\sc Theorem}[section]

\newtheorem{claim}[theorem]{\indent\sc Claim}
\newtheorem{conjecture}[theorem]{\indent\sc Conjecture}
\theoremstyle{definition} 

\def\Pi{{\mathbf{P}}}

\begin{document}

\title[Spanning trees with at most two branch vertices in claw - free graphs]{Spanning trees with at most two branch vertices in claw - free graphs} 

\author[P.~H.~Ha]{Pham Hoang Ha} 

\author[D.~D.~Hanh]{Dang Dinh Hanh} 
\keywords{spanning tree, branch vertices, claw-free.}

\subjclass[2010]{ 
Primary 05C05, 05C70. Secondary 05C07, 05C69.
}

\address{
Department of Mathematics, Hanoi National University of Education, 136, XuanThuy str., Hanoi, Vietnam
}
\email{ha.ph@hnue.edu.vn}

\address{
Department of Mathematics, Hanoi Architectural University, km10, NguyenTrai str., Hanoi, Vietnam
}
\email{ddhanhdhsphn@gmail.com}


\maketitle

\begin{abstract}
In this article, we will prove that if $G$ is a connected claw-free graph and either $\sigma_6(G)\geq |G|-5$ or $\sigma_7(G)\geq |G|-2$, here $\sigma_k(G)$ is the minimum degree sum of $k$ independent vertices in $G$, then $G$ has a spanning tree with at most two branch vertices.
\end{abstract}

\section{Introduction} 
In this article, we always consider simple graphs, which have neither loops nor multiple edges. For a graph $G$, let $V(G)$ and $E(G)$ denote the set of vertices and the set of edges of G, respectively. We write $|G|$ for the order of $G$ (i.e., $|G| = |V(G)|).$ For a vertex $v$ of $G$, we denote by $\deg_{G}(v)$ the degree of $v$ in $G.$ \\ 
 For an integer $k\geqslant 2$, we define
 $$\sigma_k(G)=\min\left\lbrace \sum_{x\in S}\deg_{G}(x):\text{for all indepentdent subset $S$ in } V(G),|S|=k\right\rbrace.$$
 In a tree, a vertex of degree one and a vertex of degree at least three is called a \textit{leaf} and a \textit{branch vertex} respectively. Many researchers have investigated the degree sum conditions for the existence of a spanning tree with a bounded number of branch vertices (see the survey article  \cite{OY} for more details).
 
 Moreover, many analogue results for the claw-free graphs are studied (see \cite{MS}, \cite{R}, \cite{GHHSV}, \cite{FKKLR} and \cite{MOY} for examples). In particular, in 2004, Gargano, Hammar, Hell, Stacho and Vaccaro gave a sufficient condition for a connected claw-free graph to have a spanning tree with few branch vertices. They proved the following theorem.
\begin{theorem}[{\cite[Gargano et al.]{GHHSV}}]\label{thm1}
Let $k$  be a non-negative integer and let $G$ be a connected claw-free graph of order $n$. If $\sigma_{k+3}\geq n-k-2$, then $G$ has a spanning tree with at most $k$ branch vertices.
\end{theorem}
After that, under the same degree condition of Theorem \ref{thm1}, Kano,
Kyaw, Matsuda, Ozeki, Saito and Yamashita showed the existence of a spanning tree with a bounded number of leaves. That seems slightly strong for the existence of a spanning tree with a
bounded number of branch vertices above. They proved the following.
\begin{theorem}[{\cite[Kano et al.]{KKMO}}]\label{theorem KKMO}
Let $k$ be a non-negative integer and let $G$ be a connected claw-free graph of order $n$. If $\sigma_{k+3}\geq n-k-2$, then $G$ has a spanning tree with at most $k+2$ leaves.
\end{theorem}
On the other hand, in 2014, Matsuda, Ozeki and Yamashita proposed the following conjecture.
\begin{conjecture}[{\cite[Matsuda et al.]{MOY}}]\label{conj1}
Let $k$ be a non-negative interger and let  $G$ be a connected claw-free graph of order $n$. If $\sigma_{2k+3}(G)\geq n-2$, then $G$ has a spanning tree with at most $k$ branch vertices. 
\end{conjecture}

In \cite{MOY}, the authors gave examples to show that Conjecture \ref{conj1} is optimal if it is correct and they also proved the conjecture while $k=1$. Motivating by the techniques in \cite{MOY}, \cite{Kyaw1} and \cite{CHH}, we would like to prove Conjecture \ref{conj1} for the case $k=2$. In particular, the main result is stated  as the following.
\begin{theorem}\label{thm-mainB}
Let $G$ be a connected claw-free graph of order $n$. If $\sigma_{7}(G)\geq n-2$, then $G$ has a spanning tree with at most two branch vertices. 	
\end{theorem}
Moreover, by using a part of the proof of Theorem \ref{thm-mainB} we also give an another result which gives an improvement of the result of Theorem \ref{thm1} while $k=3.$ We will prove the following theorem.
\begin{theorem}\label{thm-mainA}
Let $G$ be a connected claw-free graph of order $n$. If $\sigma_{6}(G)\geq n-5$, then $G$ has a spanning tree with at most two branch vertices. 
\end{theorem}

\section{Proof of Theorem \ref{thm-mainB} and Theorem \ref{thm-mainA} }
\vspace{0.5cm}

Before proving the theorems we give some notations for convenience.

Let $T$ be a spanning tree of $G.$ Setting $L(T)$ and $B(T)$ the set of leaves and the set of branch vertices of the tree $T,$ respectively. For $u, v \in V(T),$ denote by $P_T[u, v]$ the unique  path in $T$ connecting $u$ and $v.$ We assign an orientation in $P_T[u, v]$ from $u$ to $v.$

For a subset $X$ in $V(G),$ set $N(X) = \{x\in V(G) |\ xy \in E(G) \text{ for some $y \in X$}\}$ and $\deg (X) = \sum_{x\in X}\deg_G(x)$. For an integer $k \geq 1,$ we denote $N_k(X) = \{x \in V(G)\ |\ |N(x) \cap X| = k\}.$

{\bf Proof of Theorem \ref{thm-mainB}}

 Suppose that $G$ has no spanning tree with at most 2 branch vertices. Let $T$ be a spanning tree of $G.$ Then $|B(T)| \geq 3$  and  we have the following
\begin{equation*}
|L(T)|=2+\sum\limits_{v\in B(T)}(\deg_T(v)-2)\geq 2+3.(3-2)=5.
\end{equation*}
On the other hand, by Theorem \ref{theorem KKMO} we conclude that $G$ has a spanning tree with at most  6 leaves. Therefore, $G$ has a spanning tree $T$ such that $5 \leq |L(T)| \leq 6.$\\
Now we will prove Theorem  \ref{thm-mainB} by giving some contradictions in four steps.

\vspace{0.3cm}
{\bf Step 1}. If there exists a spanning tree $T$ of $G$ such that $|L(T)|= 5$ and $T$ has exactly 3 branch vertices $s, w, t$ of degree 3, where $w\in P_T[t,s]$ (see figure 1). 

\begin{figure}[h]
	\centering
	\includegraphics[width=0.4\linewidth]{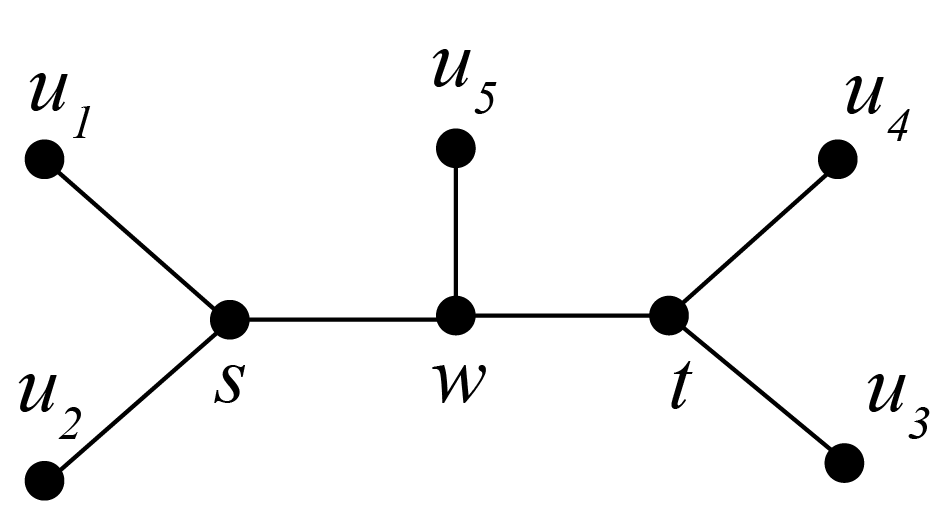}
	\caption[Tree T]{The tree $T$ is in Step 1}
	\label{Pic1}
\end{figure}

Let $L(T)=\{u_1, u_2, u_3, u_4, u_5\}$ be the set of leaves of  $T$.
Let $B_i$ be a vertex set of components of $T-\{s,w,t\}$ such that $L(T)\cap B_i=\{u_i\}$ for $1\leq i\leq 5$ and the only vertex of $N_T\{s,w,t\}\cap B_i$ is denoted by $v_i$. Without loss of generality, we may assume that $B_i\cap N_T(s) \not= \emptyset\ (1\leq i \leq 2), B_j\cap N_T(t)\not= \emptyset\ (3\leq j \leq 4)$ and $B_5\cap N_T(w)\not= \emptyset.$ Set $P_1=V(P_T[w, s]-\{w,s\}),\ P_2=V(P_T[t, w]-\{t,w\})$ and $P=P_1\cup P_2$. Set $r_1=d_T(s,t), r_2=d_T(s,w)$.  For each $x\in P_T[t, s] $ or $ P_T[s, u_i], (1\leq i \leq 2)$ or $ P_T[t, u_j](3\leq j \leq 4)$ or $P_T[w, u_5])$, its successor $x^{+}$ and the predecessor $x^{-}$ are defined, if they exist.  

We choose the tree $T$ such that:\\
(C1) $(r_1;r_2)$ is as small as possible in lexicographic order.

Set $I=\{u_1;u_2;u_3;u_4;u_5;u_6=t;u_7=s\}$.
\begin{claim}\label{claim1} We have\
	$v_1s^{-}, v_2s^{-}, v_3t^{+}, v_4t^{+}\notin E(G)$ and $v_1v_2, v_3v_4\in E(G)$.
\end{claim}
\begin{proof}
	If $v_1s^{-}\in E(G)$ then we consider the tree $T'=T+v_1s^{-}-sv_1.$ It makes a contradiction with the condition (C1) or $T'$ has two branch vertices. Hence $v_1s^{-}\notin E(G)$.\\
	Similarly, we also have \ $v_2s^{-}, v_3t^{+}, v_4t^+\notin E(G)$.\\
	Now combining with the properties of the claw-free graph $G$ we obtain $v_1v_2, v_3v_4\in E(G)$.
\end{proof}
\begin{claim}\label{claim2} 
	$I$ is an independent set.
\end{claim}
\begin{proof}
	For each $1\leq i< j\leq 5,$ if $u_iu_j\in E(G)$ then we consider the tree 	
$$T' = T+u_iu_j-v_iv_i^{-}.$$
	The resulting tree is a spanning tree of $G$ with two branch vertices, this gives a contradiction. \\
	For $i\in \{1;2;5;7\},$ if $u_iu_6\in E(G)$ then we consider the tree 	
$T'=T+u_iu_6-ww^{-}.$
	The resulting tree is a spanning tree of $G$ with two branch vertices, which is a contradiction. If $u_iu_6\in E(G)$ for some $i\in \{3; 4\}$ then by Claim \ref{claim1} the tree $T'=T+u_6u_i+v_3v_4-u_6v_3-u_6v_4$ is a spanning tree of $G$ with two branch vertices. This also gives a contradiction.\\
Similarly, we also have $u_7u_i\notin E(G)$ for all $1\leq i\leq 5$.\\
	Claim \ref{claim2} is proved.
\end{proof}

Since $G$ is claw-free and Claim \ref{claim2} holds we have $N_3(I)=\emptyset.$
\begin{claim}\label{claim3}
	$v_i\notin N(u_j)$ for all \ $1\leq i\not= j\leq 5$, and $u_6v_1, u_6v_2, u_6v_5, u_7v_3, u_7v_4, u_7v_5\not\in E(G)$.
\end{claim}
\begin{proof}
	For all  $1\leq i\not= j\leq 5$, if $v_iu_j\in E(G)$ then we consider the tree
	$$	T' =T+v_iu_j-v_iv_i^{-}.$$
	The resulting tree has two branch vertices. This gives a contradiction.\\
	Now if $u_6v_i\in E(G)$, for some $i\in\{1;2;5\}$ then the tree
	$	T' = T+u_6v_i-v_iv_i^{-} $	is a spanning tree of $G$ with two branch vertices, a contradiction.\\
Similarly, we also get $u_7v_3, u_7v_4, u_7v_5\notin E(G)$.
\end{proof}
\begin{claim}\label{claim4}
	For all $1\leq i\leq 5$, $1\leq j\leq 7$, $j\not= i,$ if $x\in B_i\cap N(u_j)$  then
\begin{enumerate}
	\item[{\rm(a)}] $x\not= u_i$,
	\item[{\rm(b)}]  $x\not= v_i$ if $j\in \{1,2,3,4,5\}$ ,
	\item[{\rm(c)}]  $x\not= v_1, v_2, v_5$ if  $j=6$, 
	\item[{\rm(d)}] $x\not= v_3, v_4, v_5$ if $j=7$, 
	\item[{\rm(e)}] $x^{-}\notin N(I-\{u_j\})$.
\end{enumerate}
\end{claim}

\begin{proof}
	By Claim \ref{claim2} and Claim \ref{claim3} we prove (a), (b), (c) and (d).\\
Now, suppose that $u_kx^{-}\in E(G)$ with $k\not= j$.\\
If $i=5$ then the tree $T'=T+u_jx+u_kx^{-}-xx^{-}-wv_5$ is a spanning tree of $G$ with two branch vertices, a contradiction.\\
Otherwise, by the same role of $s$ and $t$ we may assume that $i\in \{1;2\}$.\\
Case 1. $j\not=7$. We consider the tree
	\begin{align*}
		T' = 
		\begin{cases}
			T+u_jx+x^{-}u_k-xx^{-}-sv_i & \text { if $k\not=7$},\\
			T+u_jx+x^{-}u_7+v_1v_2-xx^{-}-sv_1-sv_2 & \text { if $k=7$}.\\
		\end{cases}
	\end{align*}
	Then the resulting tree is a spanning tree of $G$ with two branch vertices, a contradiction.\\
Case 2. $j=7$. Set $h\in\{1;2\}-\{i\}$.\\
If $k\not= h$ then the tree $T'=T+sx+u_kx^{-}+v_1v_2-xx^{-}-sv_1-sv_2$ is a spanning tree of $G$ with two branch vertices, a contradiction.\\
If $k=h$ then , since $\{ss^{-}, sv_i, sx\}$ is not claw, we have either $xs^{-}\in E(G)$ or $xv_i\in E(G)$.\\
Subcase 2.1. If $xs^{-}\in E(G)$ then the tree $T'=T+s^{-}x+u_hx^{-}-xx^{-}-sv_i$ gives a contradiction with the condiction (C1) or $T'$ has two branch vertices, this implies a contradiction.\\
Subcase 2.2. If $xv_i\in E(G)$ then we consider the tree $T'=T+xv_i+u_hx^{-}-xx^{-}-sv_i.$ This implies a contradiction from the fact that $T'$ is a spanning tree of $G$ with two branch vertices.\\
Claim \ref{claim4} is proved.
\end{proof}
\begin{claim}\label{claim5}
	We have $N_2(I-\{u_i\})\cap B_i=\emptyset$ for all $1\leq i\leq 5$.
\end{claim}
\begin{proof}
	Now, suppose that there exists $x\in N_2(I-\{u_i\})\cap B_i$. Then $xu_j, xu_k\in E(G)$ for some $j,k\not=i$. By Claim \ref{claim4} we get $x^{-}\notin N(I-\{u_j\}) \cup N(I-\{u_k\})=N(I)$. Hence, $\{xu_j, xu_k, xx^{-}\}$ is claw, which gives a contradiction. Therefore $N_2(I-\{u_i\})\cap B_i=\emptyset$.
\end{proof}
Since Claim \ref{claim4} and Claim \ref{claim5}, for $1\leq i \leq 5,$ $\{u_i\}, N(u_i)\cap B_i$ and $ (N(I-\{u_i\}))^{-}\cap B_i$ are pair-wise disjoint subsets
of $B_i$, where $(N(I-\{u_i\}))^{-}\cap B_i=\{x^{-}:\ x\in N(I-\{u_i\})\cap B_i\}$ and $ N_3(I)=(N_2(I)- N(u_i))\cap B_i=\emptyset.$ Thus, for each $i\in \{1;2;3;4\},$ we have
\begin{equation}\label{eq11}
\sum\limits_{j=1}^7 |N_G(u_j)\cap B_i|\leq |B_i|.
\end{equation}
Moreover, when $i=5$ we have 
\begin{equation}\label{eq12}
\sum\limits_{j=1}^7 |N_G(u_j)\cap B_5|\leq |B_5|-1.
\end{equation}
Now we consider the set $N(I) \cap V(P_T[t,s]-\{s,t\})$.
\begin{claim}\label{claim6}
	For all	$1\leq i\leq 4$, then $N(u_i)\cap P=\emptyset$.
\end{claim}
\begin{proof}
	For each $1\leq i\leq 4$, if there exists $y\in N(u_i)\cap P$ then we consider the tree
$$ T'= 	T+yu_i-v_iv_i^{-}.$$
This contradicts the condition (C1). Hence Claim \ref{claim6} holds.
\end{proof}
\begin{claim}\label{claim7}
	For all	$1\leq i\leq 4$, then $w\notin N(u_i)$.
\end{claim}
\begin{proof}
	For each $1\leq i\leq 4$, if there exists $wu_i\in E(G)$ the we set the tree
	$$ T'=	T+wu_i-v_iv_i^{-} .$$
Then $T'$ is a spanning tree of $G$ with two branch vertices. This gives a contradiction.
\end{proof}
\begin{claim}\label{claim8}
	$v_5w^{-}\in E(G), wu_5 \notin E(G), wu_6\notin E(G)$.
\end{claim}
\begin{proof}
	If $w^{+}v_5\in E(G)$ then consider the tree $T'=T+w^{+}v_5-wv_5$ has two branch vertices or it contradicts the condition $(C1).$ If $w^{+}w^{-}\in E(G)$ then the tree $T'=T+w^{+}w^{-}-ww^{-}$ has two branch vertices or it gives a contradiction with the condition (C1). Then, since $\{ww^{+}, ww^{-}, wv_5\}$ is not claw we obtain $v_5w^{-}\in E(G)$.\\
	For $i \in \{5; 6\},$ if $u_iw\in E(G)$ then the tree
	 $T'=T+u_iw+v_5w^{-}-v_5w-ww^{-}$ is a spanning tree of $G$ with two branch vertices, a contradiction. Hence $u_iw\notin E(G)$ for all $5\leq i \leq 6.$	
\end{proof}
By Claims \ref{claim7} and \ref{claim8} we conclude that $wu_i\notin E(G)$ for all $1\leq i\leq 6.$ Then, we get
\begin{equation}\label{eq13}
\sum\limits_{j=1}^7 |N_G(u_j)\cap \{w\}|=|N(s)\cap \{w\}|\leq 1.
\end{equation}
\begin{claim}\label{claim9} We have
\begin{equation}\label{eq14}
\sum\limits_{j=1}^7 |N_G(u_j)\cap P_1|\leq |P_1|.
\end{equation}
\end{claim}
\begin{proof}
By Claim \ref{claim6} we have
\begin{equation*}
\sum\limits_{j=1}^7 |N_G(u_j)\cap P_1|=|N(u_5)\cap P_1|+|N(s)\cap P_1|+|N(t)\cap P_1|.
\end{equation*}
 Assume that there exists $x\in N(u_5)\cap P_1$. Then we obtain a contradiction with the condiction (C1) by considering the tree $T'=T+xu_5-v_5w.$ So $N(u_5)\cap P_1=\emptyset$.\\
Now we will prove that $N(s)\cap N(t)\cap P_1=\emptyset.$ Indeed, if there exists $x\in N(s)\cap N(t)\cap P_1$. If $x^{-}t\in E(G)$ then we consider the tree $T'=T+xt+xt^{-}-xx^{-}-ww^{-}.$ Hence $T'$ is a spanning tree of $G$ with two branch vertices, a contradiction. This implies that $xt^{-}\notin E(G)$. Combining with the fact that $G$ is claw-free and considering four vertices $\{x,t,s,x^{-}\}$ we get $sx^{-}\in E(G)$. Now we consider the tree $T'=T+xt+sx^{-}-xx^{-}-ww^{-}$. Hence $T'$ is a spanning tree of $G$ with two branch vertices, which is a contradiction. Therefore, $N(s)\cap N(t)\cap P_1=\emptyset.$ This completes (\ref{eq14}).
\end{proof}
\begin{claim}\label{claim10}
	If $z\in N(u_5)\cap P_2$ then $z\not= w^{-}$ and $w^{-}\notin N(I)$.
\end{claim}
\begin{proof}
	If $z=w^-$ then the tree $T'=T+u_5w^{-}-ww^{-}$ is a spanning tree of $G$ with two branch vertices, which is a contradiction. Hence $z\not= w^{-}$.\\	
	If $w^{-}t\in E(G)$ then the tree $T'=T+u_5z+tw^{-}-ww^{-}-zz^{-}$  is a spanning tree of $G$ with two branch vertices, which gives a contradiction. Hence $w^{-}t\notin E(G)$.\\
By Claim \ref{claim6} we have $u_iw^{-}\notin E(G)$ for all $1\leq i\leq 4$.\\
If  $w^{-}s\in E(G)$ then the tree $T'=T+sw^{-}-ww^{-}$ is a spanning tree of $G$ with 2 branch vertices, this implies a contradiction. Hence $w^{-}s\notin E(G)$. Therefore $w^{-}\notin N(I)$.
\end{proof}
\begin{claim}\label{claim11}
	$|N(u_5)\cap N(t)\cap P_2|\leq 1$.
\end{claim}
\begin{proof}
Suppose that there exist $x,y\in N(u_5)\cap N(t)\cap P_2$, $x\not =y$. Without loss of generality, we may assume $y\in P_T[t,x]$.\\
Suppose that $u_5x^{-} \in E(G)$ then by $xt\in E(G)$ we consider the tree $T'=T+xt+u_5x^{-} -xx^{-}-wv_5.$ The resulting tree has two branch vertices, which is a contradiction. Hence we get $u_5x^{-}\notin E(G)$. Then $x^{-}\not= y$. Similarly, we also obtain $u_5y^{-}\notin E(G)$. Combining with $\{xu_5, xt, xx^{-}\}$ is not claw we obtain $tx^{-}\in E(G)$. By the condition (C1) it is easy to give $N(v_3)\cap P_2=\emptyset$. Then since $\{tv_3, ty, tx^{-}\}$ is not claw we obtain $yx^{-}\in E(G)$. Therefore we may conclude $x^{-}y^{-}\in E(G)$ by $\{yu_5,  yy^{-}, yx^{-}\}$ is not claw. Hence we consider the tree $T'=T+x^{-}y^{-}+u_5y-yy^{-}-wv_5$ to contradict the condition (C1).
\end{proof}
 \begin{claim}\label{claim12}
$N(s)\cap N(t)\cap P_2=\emptyset$.
\end{claim}
\begin{proof}
Assume that $x\in N(s)\cap N(t)\cap P_2$. \\
If $x^{+}t\in E(G)$ then we have the fact that the tree $T'=T+sx+tx^{+}-xx^{+}-ww^{+}$ is a spanning tree of $G$ with two branch vertices. This gives a contradiction. Hence $tx^{+}\notin E(G)$.Thus, since $\{xs,xt,xx^{+}\}$ is not claw we have $sx^{+}\in E(G)$.\\
We consider the tree $T'=T+sx+sx^{+}-xx^{+}-ww^{+}.$ Hence, $T'$ is a spanning tree of $G$ with two branch vertices, a contradiction.  Therefore, $N(s)\cap N(t)\cap P_2=\emptyset$.
\end{proof}
 \begin{claim}\label{claim13}
$N(s)\cap N(u_5)\cap P_2=\emptyset$.
\end{claim}
\begin{proof}
Assume that $x\in N(s)\cap N(u_5)\cap P_2$.\\
If $sx^{-}\in E(G)$ then we consider the tree $T'=T+sx+sx^{-}-xx^{-}-ww^{+}.$ The resulting tree is a spanning tree of $G$ with two branch vertices, a contradiction. Therefore, $sx^{-}\notin E(G)$.
Thus, since $\{xs,xu_5,xx^{-}\}$ is not claw we obtain $u_5x^{-}\in E(G)$.\\
Now, we set the tree $T'=T+u_5x^{-}+sx-xx^{-}-ww^{+}$ then $T'$ is a spanning tree of $G$ with two branch vertices, which is a contradiction. Claim \ref{claim13} is proved.
\end{proof}
By Claims \ref{claim10}-\ref{claim13}, we have $N(s)\cap N(t)\cap P_2=N(s)\cap N(u_5)\cap P_2=\emptyset$, $|N(t)\cap N(u_5)\cap P_2|\leq 1$ and $w^{-}\notin N(I)$ if $|N(t)\cap N(u_5)\cap P_2|=1.$ Hence, combining with Claim \ref{claim6}, we obtain
\begin{equation}\label{eq15}
\sum\limits_{j=1}^7 |N_G(u_j)\cap P_2|\leq |P_2|.
\end{equation}
By (\ref{eq11})-(\ref{eq15}) we conclude that
\begin{align*}
|G|&=\sum\limits_{i=1}^5|B_i|+|P_T[s,t]|\\
&\geq (1+\sum\limits_{i=1}^5\sum\limits_{j=1}^7 |N_G(u_j)\cap B_i|)+3+|P_1|+|P_2|\\
&\geq 3+\sum\limits_{i=1}^5\sum\limits_{j=1}^7 |N_G(u_j)\cap B_i|+\sum\limits_{j=1}^7 |N_G(u_j)\cap P_1|+\sum\limits_{j=1}^7 |N_G(u_j)\cap P_2|+\sum\limits_{j=1}^7 |N_G(u_j)\cap \{w\}|\\
& \geq 3+\deg (I)\geq 3+\sigma_7(G)\geq 3+(|G|-2)=1+|G|.
\end{align*}
This gives a contradiction. Step 1 is proved.

\vspace{0.5cm}
{\bf Step 2.}  $|L(T)| = 6$ and $T$ has three branch vertices $s,w,t$ with $\deg_G{(s)}=4, \deg_G(w)=\deg_G(t)=3$ and $w\in P_T[t,s]$ (see figure 2).

\begin{figure}[h]
	\centering
	\includegraphics[width=0.4\linewidth]{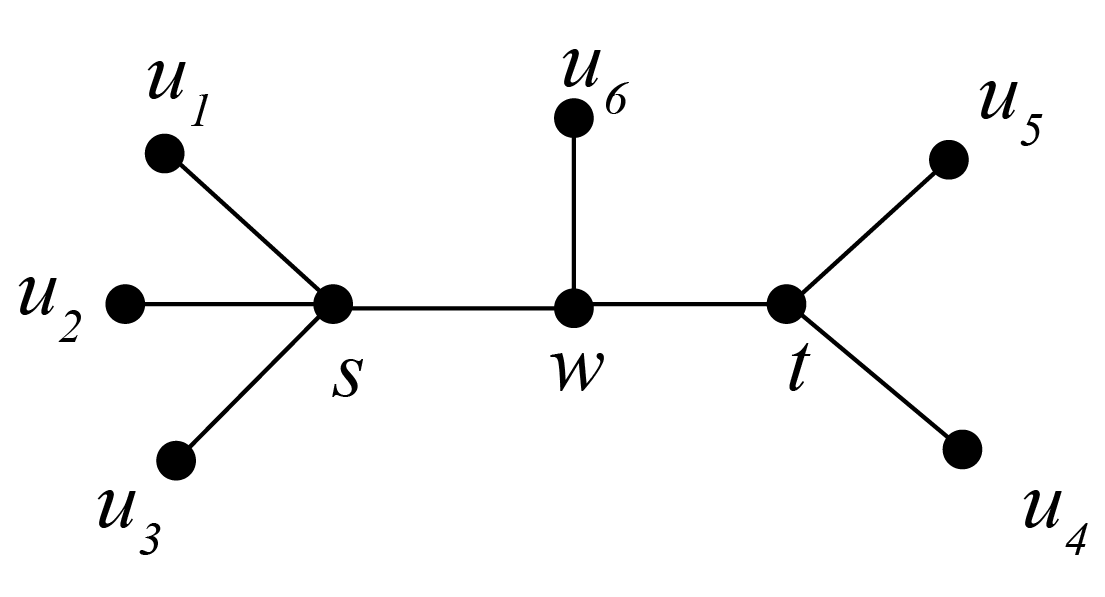}
	\caption[Tree T]{The tree $T$ is in Step 2}
	\label{Pic2}
\end{figure}

Let $L(T)=\{u_1,u_2,u_3,u_4,u_5,u_6\}$ be the set of leaves of $T$.

Let $B_i$ be a vertex set of components of $T-\{s,w,t\}$ such that $L(T)\cap B_i=\{u_i\}$ for $1\leq i\leq 6$ and the only vertex of $N_T\{s,w,t\}\cap B_i$ is denoted by $v_i$. In this step, we may assume $B_i\cap N_T(s) \not= \emptyset\ (1\leq i \leq 3), B_j\cap N_T(t)\not= \emptyset\ (4\leq j \leq 5)$ and $B_6\cap N_T(w)\not= \emptyset.$ Set $P_1=V(P_T[w, s]-\{w,s\}),\ P_2=V(P_T[t, w]-\{t,w\})$ and $P=P_1\cup P_2$. Set $r_1=d_T(s,t), r_2=d_T(s,w)$. For each $x\in P_T[t, s] $ or $ P_T[s, u_i], (1\leq i \leq 3)$ or $ P_T[t, u_j](4\leq j \leq 5)$ or $P_T[w, u_6])$, its successor $x^{+}$ and the predecessor $x^{-}$ are defined, if they exist. 
We choose the tree $T$ such that:\\
(C2)\ $(r_1;r_2)$ is as small as possible in lexicographic order.

Using similar arguments as in the proof of Claims \ref{claim1}, \ref{claim8} we have.
 \begin{claim}\label{claim14}
$t^{+}v_4, t^{+}v_5, v_6w^{+}\notin E(G)$, $v_4v_5, v_6w^{-}\in E(G)$ and $u_6w\notin E(G)$.
\end{claim}
Set $u_7=t$ and $ I=\{u_1;u_2;u_3;u_4;u_5;u_6;u_7\}$.  
 \begin{claim}\label{claim15}
$I$ is an independent set and $N_3(I)=\emptyset.$
\end{claim}
\begin{proof}
For $1\leq i<j\leq 6$, if $u_iu_j\in E(G)$ then we consider the tree 
$$	T' = 	T+u_iu_j-v_iv_i^{-}.$$
Then either the resulting tree $T'$ is a spanning tree of $G$ with two branch vertices, a contradiction, or coming back Step 1 to give a contradiction. Then  $u_iu_j\notin E(G)$ for all $1\leq i<j\leq 6.$\\
If $u_7u_j\in E(G)$, $j\in \{4;5\}$ then by Claim \ref{claim14} we can see that the tree $T'=T+u_ju_7+v_4v_5-u_7v_4-u_7v_5$ is a spanning tree of $G$ with two branch veritces, a contradiction.\\
Now if $u_7u_i\in E(G)$ for some $i\in\{1;2;3;6\}$ then the tree $T'=T+u_7u_i-ww^{-}$
is a spanning tree of $G$ with two branch vertices, a contradiction. Hence $u_7u_i\in E(G)$ for all  $i\in\{1;2;3;6\}$. Therefore, $I$ is an independent set.\\
Moreover, since $G$ is claw-free and $I$ is an independent set we obtain $N_3(I)=\emptyset.$ \\
Claim \ref{claim15} is proved.
\end{proof}

\begin{claim}\label{claim16}
	$v_i\notin N(u_j)$ for all \ $1\leq i\not= j\leq 6$, $u_7v_6\not\in E(G)$ and $\sum\limits_{i=1}^3 |N_G(u_7)\cap \{v_i\}| \leq 1$.
\end{claim}
\begin{proof}
$v_i\notin N(u_j)$ for all \ $1\leq i\not= j\leq 6$, $u_7v_6\not\in E(G)$ is proved by the similar arguments as in proof of Claim \ref{claim3}.\\
If there exist $u_7v_i, u_7v_j\in E(G)$ for some $i,j\in \{1;2;3\}, i\not=j$ then $T'=T+u_7v_i+u_7v_j-sv_i-sv_j$ is a spanning tree of $G$ with two branch vertices, a contradiction. Hence $\sum\limits_{i=1}^3 |N_G(u_7)\cap \{v_i\}| \leq 1$.
\end{proof}
Using the same arguments as in proofs of Claim \ref{claim4} and Claim \ref{claim5} we may prove the following claims.
\begin{claim}\label{claim17}
	For all $1\leq i\leq 6$, $1\leq j\leq 7$, $j\not= i,$ if $x\in B_i\cap N_G(u_j)$  then $x\not= u_i$ and $x^{-}\notin N(I-\{u_j\})$.
\end{claim}
\begin{claim}\label{claim18}
	We have $N_2(I-\{u_i\})\cap B_i=\emptyset$ for all $1\leq i\leq 6$.
\end{claim}
By Claim \ref{claim17} and Claim \ref{claim18} we firstly have
\begin{equation}\label{eq21}
\sum\limits_{i\in\{4;5\}}\sum\limits_{j=1}^7 |N_G(u_j)\cap B_i|\leq \sum\limits_{i\in\{4;5\}}|B_i|.
\end{equation}
After that, by Claim \ref{claim16} we also obtain 
\begin{equation}\label{eq22}
\sum\limits_{i\in\{1;2;3;6\}}\sum\limits_{j=1}^7 |N_G(u_j)\cap B_i|\leq \sum\limits_{i\in\{1;2;3;6\}}|B_i|-3.
\end{equation}
Since $\{sv_1,sv_2,sv_3\}$ is not claw there exist two vertices which we may assume that $v_1,v_2$ such that $v_1v_2\in E(G)$.
 \begin{claim}\label{claim19}
$N_G(u_i)\cap P_T [t,s]=\emptyset$ for $i=1;2$.
\end{claim}
\begin{proof}
Suppose that there exists $x\in N(u_i)\cap P_T[t,s]$ with $i=1;2$. \\
If either $x=w$ or $x=t$ then we consider the tree $T'=T+xu_i+v_1v_2 - sv_1-sv_2.$ The tree $T'$ is a spanning tree of $G$ with 2 branch vertices, a contradiction. \\
If either $x=s$ or $x\in P$  then we consider a new tree $T'=T+xu_i+v_1v_2-sv_1-sv_2.$ Now, using the same arguments as in Step 1 we give a contradiction.\\
Claim \ref{claim19} is proved.
\end{proof}

\begin{claim}\label{claim20} We have
\begin{equation}\label{eq23}
\sum\limits_{j=1}^7 |N_G(u_j)\cap \{s\}|=0.
\end{equation}
\end{claim}
\begin{proof}
By Claim \ref{claim19} we get $su_1; su_2\notin E(G)$.\\
If $su_i\in E(G)$ for some $i\in\{4;5;6;7\}$ then the tree $T'=T+su_i-ww^{+}$
is a spanning tree of $G$ with two branch vertices, this gives a contradiction. Hence $su_i\notin E(G)$ for all $4 \leq i \leq 7$.\\
Now, assume that $u_3s\in E(G)$. If $u_3s^{-}\in E(G)$ then we consider the tree $T'=T+u_3s^{-}-ss^{-}.$ Hence, using the same arguments as in Step 1 we give a contradiction. Therefore $u_3s^{-}\notin E(G)$. Thus, since $\{ss^{-}, su_3, sv_1\}$ is not claw, we have $s^{-}v_1\in E(G)$.\\
Repeating the same arguments we also have $s^{-}v_2\in E(G)$.\\
Now we consider the tree $T'=T+s^{-}v_1+s^{-}v_2-sv_1-sv_2.$ Then the resulting graph $T'$ has two branch vertices or this contradicts the condition (C2). Hence $u_3s\notin E(G)$.\\
Claim \ref{claim20} is completed.
\end{proof}
\begin{claim}\label{claim21} We have
\begin{equation}\label{eq24}
\sum\limits_{j=1}^7 |N_G(u_j)\cap \{w\}|\leq 1.
\end{equation}
\end{claim}
\begin{proof}
By Claim \ref{claim19} and Claim \ref{claim14} we have $wu_1; wu_2; wu_6\notin E(G)$.\\
Now for $i\in\{4;5;7\}$, then by Claim \ref{claim14} we may consider the tree $T'=T+wu_i+w^{-}v_6-ww^{-}-wv_6.$ Then $T'$  is a spanning tree of $G$ with two branch vertices, a contradiction. Then $wu_i\notin E(G)$. We thus give the following
\[\sum\limits_{j=1}^7 |N_G(u_j)\cap \{w\}|=|N(u_3)\cap \{w\}|\leq 1.\]
\end{proof}
Continuously, we will consider the set $N(I) \cap P = N(I) \cap (P_1\cup P_2)$.\\
By the condition (C2), we have.
 \begin{claim}\label{claim22}
$N(u_4)\cap P = N(u_5)\cap P=\emptyset$.
\end{claim}
\begin{claim}\label{claim23} We have
\begin{equation}\label{eq25}
\sum\limits_{j=1}^7 |N_G(u_j)\cap P_1|\leq |P_1|.
\end{equation}
\end{claim}
\begin{proof}
By Claim \ref{claim19} and Claim \ref{claim22} we have $N(u_i)\cap P_1=\emptyset$ for all $i\in \{1;2; 4; 5\}$. By the condition (C2)  we get $N(u_6)\cap P_1=\emptyset$. Hence
\[\sum\limits_{j=1}^7 |N_G(u_j)\cap P_1|=|N(u_3)\cap P_1|+|N(u_7)\cap P_1|.\]
Now, suppose that there exists $x\in N(u_3)\cap N(u_7)\cap P_1$. \\
If $u_7x^{-}\in E(G)$ then the tree $T'=T+u_7x+u_7x^{-}-xx^{-}-ww^{-}$ is a spanning tree of $G$ with two branch vertices, which is a contradiction. Hence $u_7x^{-}\notin E(G)$. Then since $\{xu_3, xx^{-}, xu_7\}$ is not claw we have $u_3x^{-}\in E(G).$ We consider the tree $T'=T+u_3x^{-}+xu_7-xx^{-}-ww^{-}.$ So $T'$ is a spanning tree of $G$ with two branch vertices, a contradiction.
Therefore, $N(u_3)\cap N(u_7)\cap P_1=\emptyset.$ Hence we get  (\ref{eq25}).\\
Claim \ref{claim23} is proved.
\end{proof}
\begin{claim}\label{claim24} We have
\begin{equation}\label{eq26}
\sum\limits_{j=1}^7 |N_G(u_j)\cap P_2|\leq |P_2|+2.
\end{equation} 
\end{claim}
\begin{proof}
By Claim \ref{claim19} and Claim \ref{claim22} we have the following
\begin{equation*}
\sum\limits_{j=1}^7 |N_G(u_j)\cap P_2|=|N(u_3)\cap P_2|+|N(u_6)\cap P_2|+|N(u_7)\cap P_2|.
\end{equation*}
Suppose that $x\in N(u_3)\cap N(u_6)\cap P_2$. If $u_6x^{-}\in E(G)$ then the tree $T'=T+u_6x^{-}+u_3x-ww^{+}-xx^{-}$ is a spanning tree of $G$ with exactly two branch vertices, a contradiction. Therefore $u_6x^{-}\notin E(G)$. Thus,  since $\{xu_3,xu_6,xx^{-}\}$ is not claw we obtain $u_3x^{-}\in E(G)$. Then $T'=T+u_6x+u_3x^{-}-xx^{-}-ww^{-}$ is a spanning tree of $G$ with exactly two branch vertices, a contradiction. We conclude that $N(u_3)\cap N(u_6)\cap P_2=\emptyset.$

Next, we prove $|N(u_3)\cap N(u_7)\cap P_2|\leq 1$. Suppose that there exist $x,y\in N(u_3)\cap N(u_7)\cap P_2$, $x\not= y$. Without loss of generality we may assume $y\in P_T[t,x]$.\\ 
If $u_3x^{-}\in E(G)$ then $T'=T+xu_7+u_3x^{-}-xx^{-}-ww^{+}$ is a spanning tree of $G$ with exactly two branch vertices, this is a contradiction. Hence $u_3x^{-}\notin E(G)$. Then since $\{xu_7, xu_3, xx^{-}\}$ is not claw we have $u_7x^{-}\in E(G)$. Now we consider the tree $T'=T+u_3y+u_7x^{-}+u_7x-xx^{-}-yy^{-}-ww^{+}$. Then the resulting graph $T'$ is a spanning tree of $G$ with exactly two branch vertices, a contradiction. So $|N(u_3)\cap N(u_7)\cap P_2|\leq 1$.

Last, we will prove $|N(u_6)\cap N(u_7)\cap P_2|\leq 1$. Suppose that there exist $x,y\in N(u_6)\cap N(u_7)\cap P_2$, $x\not= y$. Without loss of generality we may assume $y\in P_T[t,x]$. If $u_6x^{-}\in E(G)$ then $T'=T+xu_7+u_6x^{-}-xx^{-}-wv_6$ is a spanning tree of $G$ with exactly two branch vertices, a contradiction. Hence $u_6x^{-}\notin E(G)$. Then since $\{xu_7, xu_6, xx^{-}\}$ is not claw we have $u_7x^{-}\in E(G)$. Now we consider the tree $T'=T+u_6y+u_7x^{-}+u_7x-xx^{-}-yy^{-}-wv_6$. Then the resulting graph $T'$ is a spanning tree of $G$ with exactly two branch vertices, which is a contradiction. So $|N(u_6)\cap N(u_7)\cap P_2|\leq 1$.

Combining all above claims we complete Claim \ref{claim24}.
\end{proof}

Summing the inequalities (\ref{eq21})-(\ref{eq26}), it yields
\begin{align*}
|G|&=\sum\limits_{i=1}^6|B_i|+|P_T[t,s]|\\
&\geq (3+\sum\limits_{i=1}^6\sum\limits_{j=1}^7 |N_G(u_j)\cap B_i|)+3+|P_1|+|P_2|\\
&\geq 3+\sum\limits_{i=1}^6\sum\limits_{j=1}^7 |N_G(u_j)\cap B_i|+\sum\limits_{j=1}^7 |N_G(u_j)\cap \{s\}|+\sum\limits_{j=1}^7 |N_G(u_j)\cap \{w\}|\\
&+\sum\limits_{j=1}^7 |N_G(u_j)\cap P_1|+\sum\limits_{j=1}^7 |N_G(u_j)\cap P_2|\\
& \geq 3+\deg (I)\geq 3+\sigma_7(G)\geq 3+(|G|-2)=1+|G|.
\end{align*}
This is a contradiction. Step 2 is completed.

\vspace{0.5cm}
{\bf Step 3}. $T$ has two branch vertices $s$ and $t$ of degree 3 and two branchs which tough with $P_T[t,s]-\{t,s\}$ at $w$ and $z$. Without loss of generality we may assume $z\in P_T[t,w]$ (here $z$ can be $w,$ see figure 3).

\begin{figure}[h]
	\centering
	\includegraphics[width=1.0\linewidth]{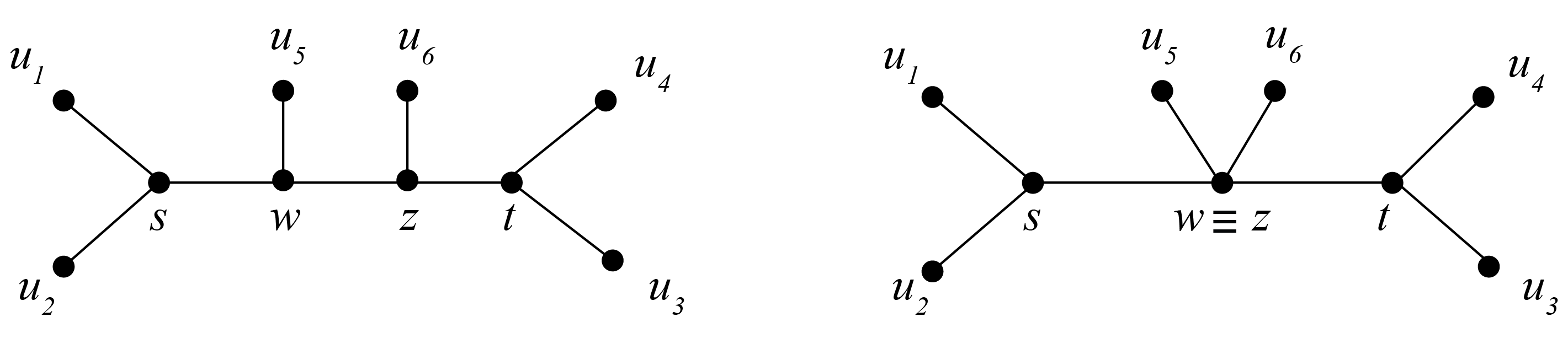}
	\caption[Tree T]{The tree $T$ is in Step 3}
	\label{Pic3}
\end{figure} 

Let $L(T)=\{u_1,u_2,u_3,u_4,u_5,u_6\}$ be the set of leaves of $T$.

Let $B_i$ be a vertex set of components of $T-\{s,w,z,t\}$ such that $L(T)\cap B_i=\{u_i\}$ for $1\leq i\leq 6$ and the only vertex of $N_T\{s,w,z,t\}\cap B_i$ is denoted by $v_i$. In this step, we may assume $B_i\cap N_T(s) \not= \emptyset\ (1\leq i \leq 2), B_j\cap N_T(t)\not= \emptyset\ (3\leq j \leq 4),$ $B_5\cap N_T(w)\not= \emptyset$ and $B_6\cap N_T(z)\not= \emptyset.$ Set $Q_1=V(P_T[w, s]-\{w,s\}), Q_2=V(P_T[z,w]-\{z,w\}),\ P_1 = Q_1 \cup Q_2,\ P_2=V(P_T[t, z]-\{t,z\})$ and $P=P_1\cup P_2$. Set $r_1=d_T(s,t), r_2=d_T(s,w),r_3=d_T(s,z) .$ For each $x\in P_T[t, s] $ or $ P_T[s, u_i], (1\leq i \leq 2)$ or $ P_T[t, u_j](3\leq j \leq 4)$ or $P_T[w, u_5]$ or $P_T[z, u_6]$, its successor $x^{+}$ and the predecessor $x^{-}$ are defined, if they exist. 

We choose the tree $T$ such that:\\
(C3) $(r_1;r_2; r_3)$ is as small as possible in lexicographic order.

Repeating the same arguments as in the proof of Claim \ref{claim1}, we have the following.
\begin{claim}\label{claim25a} We have\
$v_1s^{-}, v_2s^{-}, v_3t^{+}, v_4t^{+}\notin E(G)$ and $v_1v_2, v_3v_4\in E(G)$.
\end{claim}
\begin{claim}\label{claim25}
$N(u_i)\cap P=\emptyset$ for all $1\leq i\leq 4$.
\end{claim}
\begin{proof}
Suppose that there exists $x\in N(u_i)\cap P$ for some $1\leq i\leq 4$. We consider the tree 
$ T'=	T+xu_i-v_iv_i^{-} $
 to give a contradiction with the condition (C3).
\end{proof}
\begin{claim}\label{claim26}
$N(u_i)\cap \{s,t,w,z\}=\emptyset$ for all $1\leq i\leq 4$ and $N(u_j)\cap \{s,t\}=\emptyset$ for all $j\in \{5;6\}$.
\end{claim}
\begin{proof}
For each $i \in \{1,2,3,4\}$, if either $w \in N(u_i)$ or $z \in N(u_i)$, then without loss of generality, we assume that $wu_i \in E(G).$\\
If $w\not= z$ then we consider the tree $T'=T+wu_i-sv_i$ for the case $i\in \{1;2\}$ and $T'=T+wu_i-tv_i$ for the case $i\in \{3;4\}.$ The last case gives a contradiction with (C3) and with the first case we use the same arguments as in Step 2 to give a contradiction too.\\
If $w=z$ then the tree 
$T'=	T+wu_i-v_iv_i^{-} $
is a spanning tree of $G$ with two branch vertices, a contradiction.

Now, for each $1\leq i \leq 6$, if either $s \in N(u_i)$ or $t \in N(u_i),$ then without loss of generality, we assume that $su_i \in E(G).$ Since Claim \ref{claim25a}, we can set 
\begin{align*}
	T' = 
	\begin{cases}
		T+su_i+v_1v_2-sv_1-sv_2 & \text { if $i\in \{1;2\}$},\\
		T+su_i-ww^{+}& \text { if $i\in \{3;4;5;6\}$}.
	\end{cases}
\end{align*}
Then the resulting tree is a spanning tree of $G$ with two branch vertices if $w=z,$ a contradiction. Otherwise we use the similar arguments as in Step 1 or Step 2 to give a contradiction.\\
This completes Claim \ref{claim26}.
\end{proof}

Set $u_7=t$ and $I=\{u_1;u_2;u_3;u_4;u_5;u_6;u_7\}$.

\begin{claim}\label{claim27}
 $I$ is an independent set.
\end{claim}
\begin{proof}
	By Claim \ref{claim26} we have $u_iu_7 \not\in E(G)$ for all $i \in \{1;2;3;4;5;6\}.$ \\
If $u_iu_j\in E(G)$ where $1\leq i\not=j\leq 6$ then we consider the tree
	$	T' = T+u_iu_j-v_jv_j^{-}.$
Then we can use the arguments as in the proofs of Step 1 and Step 2 or $T'$ has two branch vertices to give a contradiction. This implies Claim \ref{claim27}.
\end{proof}
Using the similar arguments as in the proof of Claim \ref{claim4} we may obtain the following.
\begin{claim}\label{claim28}
For all $1\leq i\not= j\leq 6,$ if $x\in B_i\cap N(u_j)$ then $x\not= u_i, v_i$ and $x^{-}\notin N(I-\{u_j\}).$  \\
For all $i\in \{1;2;5;6\}, $ if $x\in B_i\cap N(u_7)$  then $x\not= u_i, v_i$ and $x^{-}\notin N(I-\{u_7\}).$ \\
For all $i\in \{3;4\},$ if $x\in B_i\cap N(u_7)$  then $x\not= u_i$ and $x^{-}\notin N(I-\{u_7\}).$ 
\end{claim}
By Claim \ref{claim28}, for $i\in \{3;4\}$ we obtain
\begin{equation}\label{eq31}
\sum\limits_{j=1}^7 |N_G(u_j)\cap B_i|\leq |B_i|.
\end{equation}
Moreover, for $i\in \{1;2;5;6\}$ we have 
\begin{equation}\label{eq32}
\sum\limits_{j=1}^7 |N_G(u_j)\cap B_i|\leq |B_i|-1.
\end{equation}
{\bf Case 1.} $z=w.$
\begin{claim}\label{claim29}
	If $z=w$ then $v_5v_6\in E(G)$, $w\notin N(u_i)$ and $N(u_i)\cap P=\emptyset$ for all $5\leq i\leq 6$.
\end{claim}
\begin{proof}
	If $w^{+}v_5\in E(G)$ the we consider the tree $T'=T+w^{+}v_5-wv_5.$ This contradicts the condition (C3). Hence $w^{+}v_5\notin E(G)$. Similarly, we also get $w^{+}v_6\notin E(G)$. Then since $\{ww^{+}, wv_5, wv_6\}$ is not claw we obtain $v_5v_6\in E(G)$ .\\
	For  $i\in \{5,6\},$ if $wu_i\in E(G)$ then we come back Step 1 with the tree $T'=T+u_iw+v_5v_6-wv_5-wv_6,$ this gives a contradiction. Hence $wu_5, wu_6\notin E(G)$. \\
	Now if there exists $x\in N(u_i)\cap P$ for some $i\in \{5,6\}$. Then we also come back Step 1 with the tree $T'=T+xu_i+v_5v_6-wv_5-wv_6.$ This gives a contradiction.\\
	Claim \ref{claim29} is proved.
\end{proof}
By Claims \ref{claim26} and \ref{claim29} we have
\begin{equation}\label{eq33}
\sum\limits_{j=1}^7 |N_G(u_j)\cap \{s,w\}|=|N(u_7)\cap \{s,w\}|= |N(t)\cap \{s,w\}|\leq 2.
\end{equation}
By Claims \ref{claim25}, \ref{claim29} we also have
\begin{equation}\label{eq34}
\sum\limits_{j=1}^7 |N_G(u_j)\cap (P_1\cup P_2)|= |N(t)\cap (P_1\cup P_2)|\leq |P_1\cup P_2|=|P_1|+|P_2|.
\end{equation}
By (\ref{eq31})-(\ref{eq34}) we obtain
\begin{align*}
|G|&=\sum\limits_{i=1}^6|B_i|+|P_T[s,t]|\\
&\geq (4+\sum\limits_{i=1}^6\sum\limits_{j=1}^7 |N_G(u_j)\cap B_i|)+3+|P_1|+|P_2|\\
&\geq 5+\sum\limits_{i=1}^6\sum\limits_{j=1}^7 |N_G(u_j)\cap B_i|+\sum\limits_{j=1}^7|N(u_j)\cap \{s,w\}|+ \sum\limits_{j=1}^7 |N_G(u_j)\cap (P_1\cup P_2)|\\
& \geq 5+\deg (I)\geq 5+\sigma_7(G)\geq 5+(|G|-2)=3+|G|.
\end{align*}
This gives a contradiction.

{\bf Case 2.} $w\not= z$.
\begin{claim}\label{claim30}
If $z\not=w$ then $w^{+}v_5, z^{+}v_6, u_5w^{-}, u_5w, u_6z^{-}, u_6z, u_6w\notin E(G)$ and $v_5w^{-}, v_6z^{-}\in E(G)$.
\end{claim}
\begin{proof}
By the condition (C3), we get $w^{+}v_5, z^{+}v_6, u_6w, w^{+}w^{-}, z^{+}z^{-}\notin E(G)$.\\
Then, since $G$ is claw-free we get $v_5w^{-}\in E(G)$ and $v_6z^{-}\in E(G)$.\\
Now, if $u_5w^{-}\in E(G)$ (or $u_6z^{-}\in E(G)$) then we may use the proof of Step 1 with the tree $T'=T+u_5w^{-}-ww^{-} ( \text{or}\ T'=T+u_6z^{-}-zz^{-}  \text{, respectively})$ to get a contradiction.\\
If $u_5w\in E(G)$ then it gives a contradiction by Step 1 when using the tree $T'=T+u_5w+v_5w^{-}-wv_5-ww^{-}.$ Hence $u_5w\notin E(G)$.\\
Similarly, we also get $u_6z\notin E(G)$. Claim \ref{claim30} is proved.
\end{proof}
\begin{claim}\label{claim31} When $z\not= w$, we have
\begin{equation}\label{eq41}
\sum\limits_{j=1}^7 |N_G(u_j)\cap \{s,w,z\}|\leq 4.
\end{equation}
\end{claim}
\begin{proof}
By Claim \ref{claim26} and Claim \ref{claim30} we have the following
\[\sum\limits_{j=1}^7 |N_G(u_j)\cap \{s,w,z\}|=|N(t)\cap\{s,w,z\}|+|N(u_5)\cap \{z\}|\leq 4.\]
\end{proof}
\begin{claim}\label{claim32} When $z\not= w$, we have
\begin{equation}\label{eq42}
\sum\limits_{j=1}^7 |N_G(u_j)\cap Q_1|\leq |Q_1|.
\end{equation}
\end{claim}
\begin{proof}
By the condition (C3) we have $N(u_5)\cap Q_1=N(u_6)\cap Q_1=\emptyset$. Combining with Claim \ref{claim25} we get
\[\sum\limits_{j=1}^7 |N_G(u_j)\cap Q_1|=|N(t)\cap Q_1|\leq |Q_1|.\]
\end{proof}
\begin{claim}\label{claim33} When $z\not= w$, we have
\begin{equation}\label{eq43}
\sum\limits_{j=1}^7 |N_G(u_j)\cap Q_2|\leq |Q_2|.
\end{equation}
\end{claim}
\begin{proof}
By the condition (C3) we have $N(u_6)\cap Q_2=\emptyset$. Combining with Claim \ref{claim25} we give the following
\[\sum\limits_{j=1}^7 |N_G(u_j)\cap Q_2|=|N(t)\cap Q_2|+|N(u_5)\cap Q_2|.\]
Firstly, we will prove that $|N(u_5)\cap N(t)\cap Q_2|\leq 1$. Indeed, suppose that there exist $x,y\in N(u_5)\cap N(t)\cap Q_2$, $x\not= y$. Without loss of generality, we may assume $x\in P_T[y,s]$.\\ 
If $u_5x^{-}\in E(G)$ then we come back previous steps with the tree $T'=T+xt+u_5x^{-}-xx^{-}-zz^{-}$ to give a contradiction. Hence $u_5x^{-}\notin E(G)$. Then since $\{xu_5, xt, xx^{-}\}$ is not claw we obtain $tx^{-}\in E(G)$. Similarly, we also get $ty^{-}\in E(G)$.\\
If $v_3x^{-}\in E(G)$ then we use the condition (C3) or the proof of Step 2 to obtain a contradiction with the tree $T'=T+x^{-}v_3-tv_3$. Hence $v_3x^{-}\notin E(G)$. Similarly, we also have $v_3y^{-}\notin E(G)$.\\
Now, since $\{tv_3, tx^{-}, ty^{-}\}$ is not claw we get $x^{-}y^{-}\in E(G)$. Then consider the tree $T'=T+x^{-}y^{-}+u_5y-x^{-}(x^{-})^{-}-yy^{-}$ to imply a contradiction with the condition (C3). So $|N(u_5)\cap N(t)\cap Q_2|\leq 1$.

Now, if $N(u_5)\cap N(t)\cap Q_2=\emptyset$ then (\ref{eq43}) holds.

If $|N(u_5)\cap N(t)\cap Q_2|=1$, set $N(u_5)\cap N(t)\cap Q_2=\{x\}$. To complete (\ref{eq43}) we will prove that $w^{-}\notin N(I)$. Indeed, by Claim \ref{claim25} and Claim \ref{claim30} we get $w^{-}\notin N(u_i)$ for all $i\in \{1;2;3;4;5\}$. Now by the condition (C3) we may obtain $w^{-}\notin N(u_6)$. On the other hand, if $w^{-}u_7\in E(G)$ then the tree $T'=T+u_5x+w^{-}u_7-ww^{-}-zz^{-}$ implies a contradiction by Step 1 or Step 2. Hence if $|N(u_5)\cap N(t)\cap Q_2|=1$ then $w^{-}\notin N(I)$. Therefore, Claim \ref{claim33} is completed.
\end{proof}
\begin{claim}\label{claim34}
When $z\not=w$, if $N(u_5)\cap P_2\not= \emptyset$ then $z^{-}\notin N(I)$.
\end{claim}
\begin{proof}
Suppose that there exists $x\in N(u_5)\cap P_2$. By Claim \ref{claim25} and Claim \ref{claim30} we have $z^{-}\notin N(u_i)$ for all $i\in\{1;2;3;4;6\}$.\\
If $u_5z^{-}\in E(G)$ then using the proof of Step 1 to give a contradiction when we consider the tree $T'=T+u_5z^{-}-zz^{-}.$ Hence $z^{-}\notin N(u_5)$.\\
If $u_7z^{-}\in E(G)$ then we consider the tree $T'=T+u_5x+u_7z^{-}-zz^{-}-xx^{-}$ to come back Step 1 when $x=t^{+}$ and Step 2 when $x\not= t^{+}.$ This implies a contradiction. Hence $z^{-}\notin N(u_7)$.\\
Therefore, Claim \ref{claim34} holds.
\end{proof}
\begin{claim}\label{claim35}
When $z\not=w$, we have $N(u_5)\cap N(u_6)\cap P_2=\emptyset$.
\end{claim}
\begin{proof}
First, we will show that $y\in N(u_j)\cap P_2$ then $y^{+}\notin N(u_k), k \not=i,$ where $\{j;k\}=\{5;6\}$. Indeed, if $y^{+}u_k\in E(G)$ then we give a contradiction as in Step 1 by considering the tree $T'=T+u_jy+u_ky^{+}-yy^{+}-zz^{-}.$ Now, suppose that there exists $y\in N(u_5)\cap N(u_6)\cap P_2$. We get $y^{+}\notin N(u_5)\cup N(u_6)$. Then $G$ contains a claw subgraph $\{yu_5, yu_6, yy^{+}\}$, a contradiction. Therefore $N(u_5)\cap N(u_6)\cap P_2=\emptyset$.
\end{proof}
\begin{claim}\label{claim36}
When $z\not=w,$ we have $|N(u_i)\cap N(t)\cap P_2|\leq 1$ for all $5\leq i \leq 6$.
\end{claim}
\begin{proof}
For $5\leq i \leq 6,$ suppose that there exist $x,y\in N(u_i)\cap N(t)\cap P_2$, $x\not= y$. Without loss of generality, we may assume that $y\in P_T[t,x]$. \\
If $u_ix^{-}\in E(G)$ then we come back Step 2 with the tree $T'=T+xt+u_ix^{-}-xx^{-}-zz^{-}$, this gives a contradiction. Hence $u_ix^{-}\notin E(G)$. Then since $\{xu_i, xt, xx^{-}\}$ is not claw we obtain $tx^{-}\in E(G)$. Similarly, we also get $ty^{-}\notin E(G)$.\\
If $v_3x^{-}\in E(G)$ then we get a contradiction with the condition (C3) with the tree $T'=T+x^{-}v_3-tv_3$. Hence $v_3x^{-}\notin E(G)$. Similarly, we also have $v_3y^{-}\notin E(G)$.\\
Now, since $\{tv_3, tx^{-}, ty^{-}\}$ is not claw we get $x^{-}y^{-}\in E(G)$. Then consider the tree 
$$	T' = 	T+x^{-}y^{-}+u_iy-v_iv_i^{-}-yy^{-},$$
which contradicts with the condition (C3). Claim \ref{claim36} is proved.
\end{proof}
By Claim \ref{claim25} and Claims \ref{claim34}-\ref{claim36} we have
\begin{equation}\label{eq44}
\sum\limits_{j=1}^7 |N_G(u_j)\cap P_2|= |N(u_5)\cap P_2|+|N(u_6)\cap P_2|+|N(t)\cap P_2|\leq |P_2|+1.
\end{equation}
Summing the inequalities (\ref{eq31}), (\ref{eq32}) and (\ref{eq41})-(\ref{eq44}), it yields
\begin{align*}
|G|&=\sum\limits_{i=1}^6|B_i|+|P_T[s,t]|\\
&\geq (4+\sum\limits_{i=1}^6\sum\limits_{j=1}^7 |N_G(u_j)\cap B_i|)+4+|P_1|+|P_2|\\
&\geq 3+\sum\limits_{i=1}^6\sum\limits_{j=1}^7 |N_G(u_j)\cap B_i|+\sum\limits_{j=1}^7|N(u_j)\cap \{s,w,z\}|+ \sum\limits_{j=1}^7 |N_G(u_j)\cap (P_1\cup P_2)|\\
& \geq 3+\deg (I)\geq 3+\sigma_7(G)\geq 3+(|G|-2)=1+|G|.
\end{align*}
This gives a contradiction with the assumptions. Step 3 is proved.

\vspace{0.5cm}
{\bf Step 4}. $|L(T)|=6$ and the tree $T$ has exactly four branch vertices of degree 3 which called $z,s,t,w$ such that $\{z\}= P_T{[t,s]} \cap P_T{[t,w]}$ (see figure \ref{Pic4}).

\begin{figure}[h]
	\centering
	\includegraphics[width=0.4\linewidth]{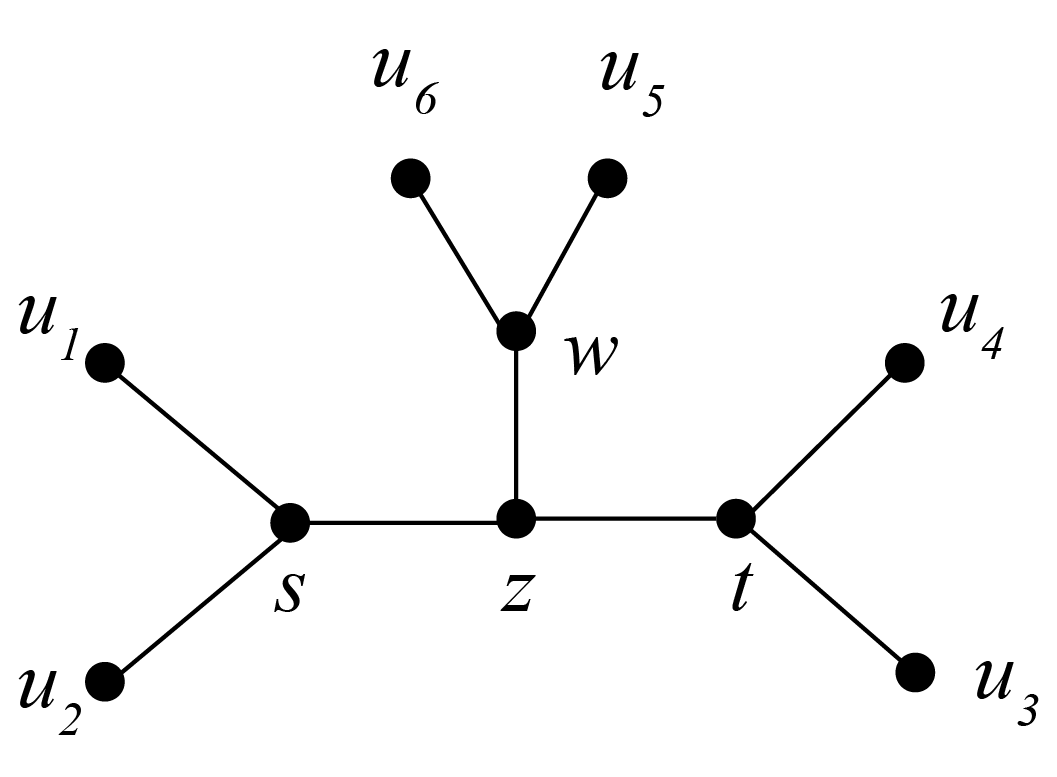}
	\caption[Tree T]{The tree $T$ is in Step 4}
	\label{Pic4}
\end{figure}

Let $L(T)=\{u_1, u_2, u_3,u_4,u_5,u_6\}$ be the set of leaves of $T.$ \\
Let $B_i$ be a vertex set of components of $T-\{s,w,z,t\}$ such that $L(T)\cap B_i=\{u_i\}$ for $1\leq i\leq 6$ and the only vertex of $N_T\{s,w,z,t\}\cap B_i$ is denoted by $v_i$. In this step, we may assume $B_i\cap N_T(s) \not= \emptyset\ (1\leq i \leq 2), B_j\cap N_T(t)\not= \emptyset\ (3\leq j \leq 4)$ and $B_k\cap N_T(w)\not= \emptyset\ (5\leq k \leq 6).$ Set $P_1=V(P_T[z,s]-\{z,s\}), P_2=V(P_T[z,t]-\{z,t\}), P_3=V(P_T[z,w]-\{z,w\})$.

We choose the tree $T$ such that:\\
(C4)\ $|P_1|+|P_2|+|P_3|$ is as small as possible.

By the condition (C4) or coming back Step 1, Step 2, Step 3 if necessary we have the following.
\begin{claim}\label{claim41}
	$v_1v_2;v_3v_4;v_5 v_6\in E(G)$ and $N(u_i)\cap (P_1\cup P_2\cup P_3\cup \{s,t,w\})=\emptyset.$
\end{claim}
Set $u_7=z,\ I=\{u_1;u_2;u_3;u_4;u_5;u_6;u_7\}$.  \\
Repeating the similar arguments as in previous steps and combining Claim \ref{claim41} we may obtain $I$ is an independent set and 
\begin{equation}\label{eq141}
\sum\limits_{j=1}^7 |N_G(u_j)\cap (P_1\cup P_2\cup P_3)|= |N(z)\cap (P_1\cup P_2\cup P_3)|\leq |P_1\cup P_2\cup P_3|=|P_1|+|P_2|+|P_3|.
\end{equation}
Moreover, using the same arguments as in previous steps and combining with Claim \ref{claim41} we also have the followings.
\begin{equation}\label{eq142}
\sum\limits_{j=1}^7 |N_G(u_j)\cap \{s,t,w\}|= |N(z)\cap \{s,t,w\}|\leq 3.
\end{equation}
\begin{equation}\label{eq143}
\sum\limits_{j=1}^7 |N_G(u_j)\cap B_i|\leq |B_i|-1 \text{ for all $i\in \{1;2; 3; 4; 5; 6\}$}.
\end{equation}
By (\ref{eq141}), (\ref{eq142}) and (\ref{eq143}) we obtain
\begin{align*}
|G|&=\sum\limits_{i=1}^6|B_i|+4+|P_1|+|P_2|+|P_3|\\
&\geq (6+\sum\limits_{i=1}^6\sum\limits_{j=1}^7 |N_G(u_j)\cap B_i|)+1+\sum\limits_{j=1}^7 |N_G(u_j)\cap \{s,t,w\}|+\sum\limits_{j=1}^7 |N_G(u_j)\cap (P_1\cup P_2\cup P_3)|\\
&\geq 7+\deg(I)\geq 7+\sigma_7(G)\geq 7+(|G|-2)=5+|G|.
\end{align*}
This is impossible. Step 4 is completed.

Finally, we complete the proof of Theorem \ref{thm-mainB}.

\vspace{0.5cm}
{\bf Proof of Theorem \ref{thm-mainA}}. 

 Suppose, to the contrary, that $G$ has no spanning tree with at most 2 branch vertices. Let $T$ be a spanning tree of $G.$ Then $|B(T)| \geq 3.$ So we have the following
\begin{equation*}
	|L(T)|=2+\sum\limits_{v\in B(T)}(\deg_T(v)-2)\geq 2+3.(3-2)=5.
\end{equation*}
On the other hand, since the assumptions of Theorem \ref{thm-mainA} we use Theorem \ref{theorem KKMO} for $k=3$ to show that $G$ has a spanning tree $T$ with at most $5$ leaves. Then $|L(T)|= 5$ and $T$ has exactly three branch vertices $s, w, t$ of degree 3. Here, we may assume that $w\in P_T[t,s]$. 

Now we use the same notations as in the proof of Step 1 of Theorem \ref{thm-mainB}. Set $X=\{u_1, u_2, u_3, u_4, u_5, u_6=t\}.$ Then $X$ is an independent set.

Repeating same arguments as in the proof of Step 1 of Theorem \ref{thm-mainB} we will obtain the followings.
\begin{equation}\label{eq51}
\sum\limits_{j=1}^6 |N_G(u_j)\cap B_i|\leq |B_i|, \text{ for all $i\in \{3; 4\}$}.
\end{equation}
\begin{equation}\label{eq52}
	\sum\limits_{j=1}^6|N_G(u_j)\cap B_i|\leq 
	|B_i|-1, \text{ for all $i\in \{1; 2;5\}$}.
\end{equation}
\begin{equation}\label{eq53}
\sum\limits_{j=1}^6 |N_G(u_j)\cap \{s,w\}|=0.
\end{equation}
\begin{equation}\label{eq54}
\sum\limits_{j=1}^6 |N_G(u_j)\cap P_1|\leq \sum\limits_{j=1}^7|N_G(u_j)\cap P_1|\leq |P_1|.
\end{equation}
\begin{equation}\label{eq55}
\sum\limits_{j=1}^6 |N_G(u_j)\cap P_2|\leq \sum\limits_{j=1}^7|N_G(u_j)\cap P_2|\leq |P_2|.
\end{equation}
By (\ref{eq51})-(\ref{eq55}) we conclude that
\begin{align*}
	|G|&=\sum\limits_{i=1}^5|B_i|+|P_T[t,s]|\\
	&\geq (3+\sum\limits_{i=1}^5\sum\limits_{j=1}^6 |N_G(u_j)\cap B_i|)+3+|P_1|+|P_2|\\
	&\geq 6+\sum\limits_{i=1}^5\sum\limits_{j=1}^6 |N_G(u_j)\cap B_i|+\sum\limits_{i=1}^6 |N_G(u_j)\cap P_T[t,s]|\\
	& \geq 6+\deg (X)\geq 6+\sigma_6(G)\geq 6+(|G|-5)=1+|G|.
\end{align*}
This gives a contradiction.\\
Theorem \ref{thm-mainA} is proved.

{\bf Acknowledgements.} The research is supported by the NAFOSTED Grant of Vietnam (No.101.04-2018.03).


\begin{thebibliography}{99}
	
\bibitem{CHH}
\textsc{Y. Chen, P. H. Ha and D. D. Hanh}, \emph{Spanning trees with at most 4 leaves in $K_{1,5}$-free graphs}, arXiv:1804.09332.

\bibitem{FKKLR} 
\textsc{E. Flandrin, T. Kaiser, R. Ku\v {z}el, H. Li, Z. Ryj\'{a}\v {c}ek}, \emph{Neighborhood unions and extremal spanning trees}, Discrete Math., \textbf{308} (2008), 2343-2350.	

\bibitem{KKMO}
\textsc{M. Kano, A. Kyaw, H. Matsuda, K. Ozeki, A. Saito and T. Yamashita}, \emph{Spanning trees with a small number of leaves in a claw-free graph}, Ars Combin. {\bf 103}, 137-154 (2012).

\bibitem{Kyaw1}
\textsc{A. Kyaw}, \emph{Spanning trees with at most 3 leaves in $K_{1,4}$-free graphs}, Discrete Math., \textbf{309} (2009), 6146 - 6148. 
		
\bibitem{GHHSV}
\textsc{L. Gargano, M. Hammar, P. Hell, L. Stacho and U. Vaccaro},
\emph{Spanning spiders and light-splitting
	switches}, Discrete Math., \textbf{285} (2004), 83 - 95.

\bibitem{MOY}
\textsc{H. Matsuda, K. Ozeki and T. Yamashita},
\emph{Spanning trees with a bounded number of branch vertices in a claw-free graph}, Graphs Combin., \textbf{30} (2014), 429-437.
		%
\bibitem{MS}
\textsc{M. M. Matthews and D. P. Sumner},
\emph{Longest paths and cycles in $K_{1,3}$ graphs}, J. Graph Theory, \textbf{9} (1985), 269-277.

\bibitem{OY}
\textsc{K. Ozeki and T. Yamashita}, 
\emph{Spanning trees: A survey}, Graphs Combin., \textbf{22} (2011), 1-26. 

\bibitem{R}
\textsc{Z. Ryj\'{a}\v {c}ek}, 
\emph{On a closure concept in claw-free graphs}, J. Combin. Theory Ser. B, \textbf{70} (1997), 217-224. 

	
	
\end{thebibliography}
\end{document}